\documentclass[12pt]{amsart}

\usepackage{enumerate, amsmath, amsthm, amsfonts, amssymb, xy,  mathrsfs, graphicx, paralist}
\usepackage[usenames, dvipsnames]{xcolor}
\usepackage[margin=1in]{geometry} 
\usepackage[bookmarks, colorlinks=true, linkcolor=blue, citecolor=blue, urlcolor=blue]{hyperref}

\input xy
\xyoption{all}

\numberwithin{equation}{section}
\newtheorem{theorem}[equation]{Theorem}

\newtheorem{lemma}[equation]{Lemma}

\theoremstyle{definition}
\newtheorem{rmk}[equation]{Remark}
\newenvironment{remark}[1][]{\begin{rmk}[#1] \pushQED{\qed}}{\popQED \end{rmk}}
\newtheorem{eg}[equation]{Example}

\newtheorem{defn}[equation]{Definition}

\newcommand{\rB}{\mathrm{B}}

\newcommand{\bC}{\mathbf{C}}

\newcommand{\rC}{\mathrm{C}}

\newcommand{\rD}{\mathrm{D}}

\newcommand{\bH}{\mathbf{H}}

\newcommand{\bO}{\mathbf{O}}

\newcommand{\bS}{\mathbf{S}}

\newcommand{\rU}{\mathrm{U}}

\newcommand{\rX}{\mathrm{X}}

\newcommand{\bZ}{\mathbf{Z}}

\newcommand{\fg}{\mathfrak{g}}

\newcommand{\bh}{\mathbf{h}}

\renewcommand{\phi}{\varphi}

\newcommand{\arxiv}[1]{\href{http://arxiv.org/abs/#1}{{\tt arXiv:#1}}}

\makeatletter
\def\Ddots{\mathinner{\mkern1mu\raise\p@
\vbox{\kern7\p@\hbox{.}}\mkern2mu
\raise4\p@\hbox{.}\mkern2mu\raise7\p@\hbox{.}\mkern1mu}}
\makeatother

\renewcommand{\hom}{\operatorname{Hom}}

\DeclareMathOperator{\rank}{rank}

\DeclareMathOperator{\Sym}{Sym}

\newcommand{\Sp}{\mathbf{Sp}}

\newcommand{\fso}{\mathfrak{so}}

\usepackage{mathtools}
\usepackage{tikz}
\usetikzlibrary{positioning}
\usetikzlibrary{chains}

\tikzset{node distance=2em, ch/.style={circle,draw,on chain,inner sep=2pt},chj/.style={ch,join},every path/.style={shorten >=4pt,shorten <=4pt},line width=1pt,baseline=-1ex}

\newcommand{\mlabel}[1]{ \footnotesize \(#1\) }
\newcommand{\dnode}[2][chj]{ \node[#1,label={below:\mlabel{#2}}] {}; }
\newcommand{\dnodenj}[1]{ \dnode[ch]{#1} }
\newcommand{\dnodebr}[1]{ \node[chj,label={below right:\mlabel{#1}}] {}; }
\newcommand{\dydots}{ \node[chj,draw=none,inner sep=1pt] {\dots}; }

\title[Jacobi--Trudi determinants and characters of minimal affinizations]{Jacobi--Trudi determinants and\\ characters of minimal affinizations}
\author{Steven V Sam}
\address{Department of Mathematics, University of California, Berkeley, CA}
\email{svs@math.berkeley.edu}
\urladdr{\url{http://math.berkeley.edu/~svs/}}
\thanks{The author was supported by a Miller research fellowship.}

\date{September 3, 2014}

\begin{document}

\maketitle

\begin{abstract}
In their study of characters of minimal affinizations of representations of orthogonal and symplectic Lie algebras, Chari and Greenstein conjectured that certain Jacobi--Trudi determinants satisfy an alternating sum formula. In this note, we prove their conjecture and slightly more. The proof relies on some symmetries of the ring of symmetric functions discovered by Koike and Terada. Using results of Hernandez, Mukhin--Young, and Naoi, this implies that the characters of minimal affinizations in types B, C, and D are given by a Jacobi--Trudi determinant.
\end{abstract}

\section*{Introduction}

In \cite{minaff}, Chari and Greenstein study a class of modules over the current algebra $\fg \otimes \bC[t]$ where $\fg$ is either a special orthogonal or symplectic Lie algebra (over the complex numbers). These modules are related to the minimal affinizations, a class of irreducible representations for the quantum loop algebra $\rU_q(\fg \otimes \bC[t,t^{-1}])$. We refer the reader to \cite[\S 3]{minaff} for background and references. A character formula, which is similar to a Jacobi--Trudi determinant, for these modules is conjectured in \cite[Conjecture 1.13]{minaff}. This is inspired by \cite{NN1} which conjectures that the characters of minimal affinizations are given by such determinants (see also \cite{NN2,NN3} for related work).

The aim of this note is to prove \cite[Conjecture 1.13]{minaff} (see Theorem~\ref{thm:main}). We will give a uniform proof for all types. The conjecture reduces to a combinatorial statement about characters of $\fg$, so we will not need to discuss current or loop algebras any further. In fact, we will prove an extension of the combinatorial statement which removes a restriction on the highest weights considered. Furthermore, using results of Hernandez, Mukhin--Young, and Naoi, this gives a character formula for minimal affinizations of representations of $\fg$ in types B, C, and D (see Remark~\ref{rmk:minaff}).

The method of proof involves passing to a suitable limit (with respect to the rank of the Lie algebra) to take advantage of additional symmetries. This suggests that there should be a connection to the categories ${\rm Rep}(\bO)$ and ${\rm Rep}(\Sp)$ studied in \cite[\S 4]{infrank} and a suitable categorification of the involutions $i_\bO$ and $i_\Sp$ used in \S\ref{sec:proof} (which were introduced by Koike and Terada in \cite{koiketerada}), but we have been unable to find one so far.

\section{Notation}

We need some basic terminology of partitions \cite[\S I.1]{macdonald}. A partition $\lambda$ is a sequence of integers $(\lambda_1, \dots, \lambda_r)$ with $\lambda_1 \ge \cdots \ge \lambda_r \ge 0$. We set $|\lambda| = \sum_i \lambda_i$ and $\ell(\lambda) = \max\{i \mid \lambda_i \ne 0\}$. We write $\mu \subseteq \lambda$ if $\mu_i \le \lambda_i$ for all $i$ and also say that $\lambda$ contains $\mu$. The notation $a^b$ means the sequence $(a, a, \dots, a)$ where $a$ appears $b$ times. We use $\lambda^\dagger$ to denote the transpose partition of $\lambda$, i.e., $\lambda^\dagger_i = \#\{j \mid \lambda_j \ge i\}$ (in terms of Young diagrams, we are flipping across the diagonal). Let $\bS_\lambda$ denote the corresponding Schur functor \cite[\S 6.1]{fultonharris}; for the purposes of this note, $\bS_\lambda$ is a functor from the category of complex vector spaces to itself. Special cases are symmetric powers $\bS_k = \Sym^k$ and exterior powers $\bS_{(1^k)} = \bigwedge^k$. We use $s_\lambda$ to denote the Schur function indexed by $\lambda$ \cite[\S I.3]{macdonald} (it is the character of $\bS_\lambda$). The product of two Schur functions is a linear combination of Schur functions: 
\[
s_\mu s_\nu = \sum_\lambda c^\lambda_{\mu, \nu} s_\lambda.
\]
The $c^\lambda_{\mu, \nu}$ are the Littlewood--Richardson coefficients \cite[\S I.9]{macdonald}. If $c^\lambda_{\mu, \nu} \ne 0$, then $|\lambda| = |\mu| + |\nu|$ and also $\mu \subseteq \lambda$ and $\nu \subseteq \lambda$.

Let $G$ be a complex classical group of type $\rB_n$, $\rC_n$, or $\rD_{n+1}$, i.e., $G$ is either $\bO_{2n+1}(\bC)$, $\Sp_{2n}(\bC)$, or $\bO_{2n+2}(\bC)$, respectively. Let $\fg$ be the Lie algebra of $G$. Let $\rank(\fg)$ be the rank of $\fg$, i.e., it is $n$ in the cases of type B and C, and it is $n+1$ in the case of type D. We use these groups rather than their Lie algebras to avoid having to make technical remarks later. For the representations considered in \cite{minaff}, this choice will not be important. We number the nodes of the Dynkin diagram according to Bourbaki notation:
\begin{align*}
\rB_n: &&
\begin{tikzpicture}[start chain]
\dnode{1}
\dnode{2}
\dydots
\dnode{n-2} 
\dnode{n-1}
\dnodenj{n}
\path (chain-5) -- node[anchor=mid] {\(\Rightarrow\)} (chain-6);
\end{tikzpicture}
\\%
\rC_n: &&
\begin{tikzpicture}[start chain]
\dnode{1}
\dnode{2}
\dydots
\dnode{n-2}
\dnode{n-1}
\dnodenj{n}
\path (chain-5) -- node[anchor=mid] {\(\Leftarrow\)} (chain-6);
\end{tikzpicture}
\\%
\rD_{n+1}: &&
\begin{tikzpicture}
\begin{scope}[start chain]
\dnode{1}
\dnode{2}
\node[chj,draw=none] {\dots};
\dnode{n-2}
\dnode{n-1}
\dnode{n}
\end{scope}
\begin{scope}[start chain=br going above]
\chainin(chain-5);
\dnodebr{n+1}
\end{scope}
\end{tikzpicture}
\end{align*}
Let $\omega_i$ be the fundamental weights, and let $\lambda$ be a dominant integral weight which is a linear combination of $\omega_1, \dots, \omega_{n-1}$ (so in particular, we avoid the spin representations in the orthogonal case). We will use a basis $e_1, \dots, e_{\rank(\fg)}$ for the weight lattice of $G$ (see \cite[\S\S 16.1, 18.1]{fultonharris} for details; there the basis is denoted $L_1, \dots, L_{\rank(\fg)}$). Given $\lambda = a_1 \omega_1 + \cdots + a_{n-1} \omega_{n-1}$, we associate to it the partition 
\[
(a_1 + \cdots + a_{n-1}, a_2 + \cdots + a_{n-1}, \dots, a_{n-1}).
\]
So in particular, the notation $\lambda_i = a_i + \cdots + a_{n-1}$ is defined. Then we have $\lambda = \lambda_1 e_1 + \cdots + \lambda_{n-1} e_{n-1}$. Let $V_\lambda$ be the corresponding highest weight representation of $G$. We will denote $V = V_1$, the vector representation. We sometimes use the notation $V_\lambda^\bO$ or $V_\lambda^\Sp$ to emphasize that we are dealing with the  orthogonal or symplectic case, respectively. 

In general, all finite-dimensional irreducible representations $V_\lambda$ of $G$ can be indexed by partitions $\lambda$ (see \cite[\S\S 17.3, 19.5]{fultonharris} or \cite[\S 4.1]{infrank}). We may assume that $\ell(\lambda) \le \rank(\fg)$ as long as we are ambivalent about the presence of the sign representation in the orthogonal group case. (The reason we do not use the special orthogonal group is because some irreducible representations of the even orthogonal group are not irreducible when restricted to the special orthogonal group, and so the latter group does not behave as well from the perspective of stability.)

Now we rephrase the definitions in \cite[\S 1.13]{minaff} in this notation. First, we have $i_\lambda = \ell(\lambda)$. In the orthogonal case, $\Psi_\lambda = \{ e_i + e_j \mid 1 \le i < j \le \ell(\lambda)\}$, and in the symplectic case, $\Psi_\lambda = \{ e_i + e_j \mid 1 \le i \le j \le \ell(\lambda)\}$. Define the set 
\[
\Gamma(\lambda, \Psi_\lambda) = \{ (\mu,s) \mid \lambda = \mu + \sum_{\beta \in \Psi_\lambda} n_\beta \beta,\ n_\beta \in \bZ_{\ge 0},\ \sum_{\beta \in \Psi_\lambda} n_\beta = s\}.
\]
By the definitions of $\Psi_\lambda$, we see that $(\mu, s) \in \Gamma(\lambda, \Psi_\lambda)$ implies that $s = (|\lambda|-|\mu|)/2$.

Define $\bh_k = {\rm char}(V_k^\bO)$ in the orthogonal case and $\bh_k = \sum_{0 \le r \le k/2} {\rm char}(V_{k-2r}^\Sp)$ in the symplectic case. In both cases, define the Jacobi--Trudi determinant
\[
\bH_\lambda = \det(\bh_{\lambda_i - i + j}).
\]
For $(\nu, s) \in \Gamma(\lambda, \Psi_\lambda)$, define 
\begin{align*}
C^\lambda_{\nu, s} = \dim \hom_G(V_\nu, \bigwedge^s(\fg) \otimes V_\lambda)
\end{align*}
(see \cite[\S 2.7]{minaff}, but there it is $c$ instead of $C$; we use $c$ for Littlewood--Richardson coefficients).

All of the above definitions make sense for any partition $\lambda$ with $\ell(\lambda) \le \rank(\fg)$. To make this clear, we spell out the conversion between partitions and weights now. Let $r = \rank(\fg)$ and let $\lambda = (\lambda_1, \dots, \lambda_r)$ be a partition.
\begin{compactitem}
\item If $G = \Sp_{2r}(\bC)$, then $V_\lambda$ is irreducible with highest weight $\sum_{i=1}^{r-1} (\lambda_i - \lambda_{i+1}) \omega_i + \lambda_r \omega_r$.
\item If $G = \bO_{2r+1}(\bC)$, then $V_\lambda$ is irreducible with highest weight $\sum_{i=1}^{r-1} (\lambda_i - \lambda_{i+1}) \omega_i + 2\lambda_r \omega_r$.
\item If $G = \bO_{2r}(\bC)$, then there are two cases. In both cases, $V_\lambda$ is an irreducible representation of $\bO_{2r}(\bC)$, but we distinguish between what happens when we pass to the Lie algebra $\fso_{2r}(\bC)$.
\begin{compactitem}
\item If $\lambda_r=0$, then $V_\lambda$ is an irreducible representation of $\fso_{2r}(\bC)$ with highest weight
$\sum_{i=1}^{r-2} (\lambda_i - \lambda_{i+1}) \omega_i + \lambda_{r-1} (\omega_{r-1} + \omega_r)$.
\item If $\lambda_r>0$, then as a representation of $\fso_{2r}(\bC)$, $V_\lambda$ is the direct sum of irreducible representations with highest weights $\sum_{i=1}^{r-2} (\lambda_i - \lambda_{i+1}) \omega_i + (\lambda_{r-1} - \lambda_r) \omega_{r-1} + (\lambda_{r-1} + \lambda_r)\omega_r$ and $\sum_{i=1}^{r-2} (\lambda_i - \lambda_{i+1}) \omega_i + (\lambda_{r-1} + \lambda_r) \omega_{r-1} + (\lambda_{r-1} - \lambda_r)\omega_r$.
\end{compactitem}
\end{compactitem}

In the orthogonal case, let $d^\lambda_\nu$ be the multiplicity of $V_\nu^{\Sp}$ in $\bS_\lambda(V^\Sp)$: here $V^\Sp$ is the vector representation for $\Sp(2n)$ with $n \ge \ell(\lambda)$ and $\bS_\lambda(V^\Sp)$ is considered as a representation of $\Sp(2n)$). By \cite[Proposition 1.5.3]{koiketerada}, this multiplicity is independent of $n$ as long as $n \ge \ell(\lambda)$, and we have
\[
d^\lambda_\nu = \sum_\eta c^\lambda_{\nu, (2\eta)^\dagger}.
\]

Similarly, in the symplectic case, let $d^\lambda_\nu$ be the multiplicity of $V_\nu^\bO$ in $\bS_\lambda(V^\bO)$ (note that we are using branching rules for the {\it other} group in both cases). Then we have
\[
d^\lambda_\nu = \sum_\eta c^\lambda_{\nu, 2\eta}.
\]

We can now formulate our main result. When $\ell(\lambda) \le n-1$, this proves \cite[Conjecture 1.13]{minaff}.

\begin{theorem} \label{thm:main}
Let $\lambda$ be a partition with $\ell(\lambda) \le \rank(\fg)$. Then 
\begin{align} \label{eqn:conj}
\sum_{(\nu, s) \in \Gamma(\lambda, \Psi_\lambda)} (-1)^s C^\lambda_{\nu, s} \bH_\nu = {\rm char}(V_\lambda).
\end{align}
Also $\bH_\lambda = \sum_\nu d^\lambda_\nu {\rm char}(V_\nu)$.
\end{theorem}

\begin{remark} \label{rmk:minaff}
Under the restriction $\ell(\lambda) \le n-1$, Chari and Greenstein constructed modules $P(\lambda,0)^{\Gamma(\lambda, \Psi_\lambda)}$ in \cite{minaff}, and Theorem~\ref{thm:main} together with \cite[Theorem 2]{minaff} shows that its character is $\bH_\lambda$. In types B and C, Naoi shows \cite[Remark 4.7]{naoi} that these modules are the ``graded limits'' of the minimal affinizations of the corresponding simple modules $V_\lambda$ of $\fg$. A similar result is obtained for a special class of highest weights in type D in \cite{naoi-D}. In particular, the characters (considered as representations of $\fg$) of both modules are the same. So the character of the minimal affinization is also $\bH_\lambda$. In type B, this follows from work of Hernandez \cite{hernandez} (see \cite[Remark 4.7]{naoi}), and also from work of Mukhin--Young \cite[Corollary 7.6]{mukhin-young}.
\end{remark}

\section{Some identities}

Let $Q_{-1}$ be the set of partitions with the following inductive definition. The empty partition belongs to $Q_{-1}$.  A non-empty partition $\mu$ belongs to $Q_{-1}$ if and only if the number of rows in $\mu$ is one more than the number of columns, i.e., $\ell(\mu)=\mu_1+1$, and the partition obtained by deleting the first row and column of $\mu$, i.e., $(\mu_2-1, \ldots, \mu_{\ell(\mu)}-1)$, belongs to $Q_{-1}$. The first few partitions in $Q_{-1}$ are $0$, $(1,1)$, $(2,1,1)$, $(2,2,2)$. Define $Q_1 = \{\lambda \mid \lambda^\dagger \in Q_{-1}\}$. We record this definition as the following formula:
\begin{align} \label{eqn:koszul-transpose}
Q_1^\dagger = Q_{-1}.
\end{align}

The significance of these sets are the following decompositions (see \cite[I.A.7, Ex.~4,5]{macdonald}):
\begin{align} 
\bigwedge^{i}(\Sym^2(E)) &= \bigoplus_{\substack{\mu \in Q_1\\ |\mu| = 2i}} \bS_{\mu}(E),\\
\label{eqn:wedge-plethysm}
\bigwedge^{i}(\bigwedge^2(E)) &= \bigoplus_{\substack{\mu \in Q_{-1}\\ |\mu|=2i}} \bS_{\mu}(E).
\end{align}

We need two of Littlewood's identities \cite[Proposition 1.5.3]{koiketerada}:
\begin{align} \label{eqn:lwood-orth}
{\rm char}(V_\lambda^\bO) &= \sum_{\mu \in Q_{1}} (-1)^{|\mu|/2} \sum_\nu c^\lambda_{\mu, \nu} s_\nu,\\
\label{eqn:lwood-symp}
{\rm char}(V_\lambda^\Sp) &= \sum_{\mu \in Q_{-1}} (-1)^{|\mu|/2} \sum_\nu c^\lambda_{\mu, \nu} s_\nu.
\end{align}

\begin{lemma} \label{lem:calc-C}
Fix $(\nu, s) \in \Gamma(\lambda ,\Psi_\lambda)$ where $\ell(\lambda) \le \rank(\fg)$ and $s = (|\lambda|-|\nu|)/2$. Then $C^\lambda_{\nu, s} = \sum_{\mu \in Q_{-1}} c^\lambda_{\mu, \nu}$ in the orthogonal case (for the symplectic case, use $Q_1$ instead of $Q_{-1}$). 

Conversely, if this sum is nonzero, then $(\nu, s) \in \Gamma(\lambda, \Psi_\lambda)$ for $s = (|\lambda|-|\nu|)/2$.
\end{lemma}

\begin{proof}
In the orthogonal case, we have $\fg = V_{1,1} = \bigwedge^2(V)$. So we need to calculate the multiplicity of $V_\nu$ in $\bigwedge^s(\bigwedge^2(V)) \otimes V_\lambda$ where $s = (|\lambda| - |\nu|)/2$. By \eqref{eqn:wedge-plethysm}, we get $\bigwedge^s(\bigwedge^2(V)) = \bigoplus_{\substack{\mu \in Q_{-1}\\|\mu| = 2s}} \bS_\mu (V)$. (In the symplectic case we instead have $\fg = V_2 = \Sym^2(V)$, so all of the following statements will hold if we replace $Q_{-1}$ with $Q_1$.) We claim that the multiplicity of $V_\nu$ in $\bS_\mu(V) \otimes V_\lambda$ is the Littlewood--Richardson coefficient $c^\lambda_{\mu, \nu}$.

If $\ell(\mu) \le \rank(\fg)$, then as a representation of the orthogonal group (also in the symplectic case), $\bS_\mu(V)$ is the sum of $V_\mu$ and other $V_\alpha$ where $|\alpha| < |\mu|$ up to twisting $V_\alpha$ with a sign character (this follows from the explicit formula in \cite[Proposition 2.5.1]{koiketerada}). Also, if $V_\nu$ appears in $V_\alpha \otimes V_\lambda$, then we must have $|\nu| \ge |\lambda| - |\alpha|$ by a basic argument with weights. This implies that the multiplicity of $V_\nu$ in $\bS_\mu(V) \otimes V_\lambda$ is the same as the multiplicity of $V_\nu$ in $V_\mu \otimes V_\lambda$ under our hypothesis that $|\nu| + |\mu| = |\lambda|$. Furthermore, the multiplicity in this case is the Littlewood--Richardson coefficient $c^\lambda_{\mu, \nu}$ \cite[Proposition 2.5.2]{koiketerada}. 

If $\ell(\mu) > \rank(\fg)$, then the multiplicity of $V_\nu$ in $\bS_\mu(V) \otimes V_\lambda$ is $0$ since all $V_\alpha$ in $\bS_\mu(V)$ satisfy $|\alpha|<|\mu|$. Also, $c^\lambda_{\mu, \nu} = 0$ since $\mu \not\subseteq \lambda$. This proves the claim and the second sentence of the lemma.

Now we handle the last sentence of the lemma. So suppose that $c^\lambda_{\mu, \nu} \ne 0$ for some $\mu \in Q_{-1}$. Set $s = (|\lambda|-|\nu|)/2 = |\mu|/2$. The weights of $\bS_\mu(V) \subset \bigwedge^s(\fg)$ are linear combinations of $s$ roots of $\fg$. In particular, $\lambda$ is the sum of $\nu$ and $s$ roots $\alpha_1, \dots, \alpha_s$ of $\fg$. The possible roots of $\fg$ are $e_i \pm e_j$ and $\pm e_i$. Since $|\nu + e_i - e_j| = |\nu|$ and $|\nu \pm e_i| = |\nu|\pm 1$, the $s$ roots $\alpha_1, \dots, \alpha_s$ must all be of the form $e_i + e_j$, so $(\nu, s) \in \Gamma(\lambda, \Psi_\lambda)$.
\end{proof}

\section{Proof of main theorem} \label{sec:proof}

\begin{lemma}
Pick $\rX  \in \{\rB, \rC, \rD\}$. Fix a partition $\lambda$ with $\ell(\lambda) \le n$. Then \eqref{eqn:conj} is true for the representation $V_\lambda$ for $\rX_n$ if and only if it is true for the representation $V_\lambda$ for $\rX_m$ for any $m \ge n$. 
\end{lemma}

\begin{proof}
By \cite[Corollary 2.5.3]{koiketerada}, the tensor product decomposition $V_\lambda \otimes V_\mu$ is independent of $m$ if $m \ge \ell(\lambda) + \ell(\mu)$, and in this case, the tensor product decomposes as a sum of $V_\alpha$ with $\ell(\alpha) \le \ell(\lambda) + \ell(\mu)$. The definition of $\bH_\lambda$ involves multiplying at most $\ell(\lambda) \le m$ characters, all indexed by one-row partitions, so its definition is independent of $m$. Certainly the set $\Gamma(\lambda, \Psi_\lambda)$ does not depend on $m$ if $m \ge \ell(\lambda)$. So it remains to show that the coefficients $C^\lambda_{\nu, s}$ are independent of $m$, but this follows from Lemma~\ref{lem:calc-C}.
\end{proof}

In particular, we may assume that $n = \infty$. In this limit, we can use some additional symmetries of the character ring $\Lambda$ of $\fg$. Then $\Lambda$ is the ring of symmetric functions, but is equipped with a new basis which was studied in \cite{koiketerada}. Write $s_{[\lambda]} = {\rm char}(V_\lambda)$. We use $s^\Sp_{[\lambda]}$ or $s^\bO_{[\lambda]}$ if we need to emphasize the group. Then the $s_{[\lambda]}$, as $\lambda$ ranges over all partitions, forms a basis for this character ring. The idea is to use \eqref{eqn:lwood-orth} or \eqref{eqn:lwood-symp} to exhibit the change of basis between $s_{[\lambda]}$ and the usual Schur functions $s_\mu = {\rm char}(\bS_\mu(V))$. There is an involution (which is an algebra automorphism), denoted $i_\bO$ in the orthogonal case and $i_{\Sp}$ in the symplectic case, that sends $s_{[\lambda]}$ to $s_{[\lambda^\dagger]}$ \cite[Theorem 2.3.4]{koiketerada}. Also, we recall that the linear map $\omega \colon s_\lambda \mapsto s_{\lambda^\dagger}$ is an algebra automorphism \cite[\S I.3]{macdonald}. We need the following identity \cite[Theorem 2.3.2]{koiketerada}
\begin{align} \label{eqn:duality}
\omega(s^\Sp_{[\lambda]}) = s^\bO_{[\lambda^\dagger]}.
\end{align}

\begin{lemma} \label{lem:schurH}
The involution $i_{\bO}$ or $i_{\Sp}$ sends $\bH_\nu$ to the Schur function $s_{\nu^\dagger}$.
\end{lemma}

\begin{proof}
In the orthogonal case, $i_{\bO}(\bh_k) = s_{[1^k]} = {\rm char}(\bigwedge^k V) = s_{1^k}$, and in the symplectic case, $i_{\Sp}(\bh_k) = \sum_{0 \le r \le k/2} s_{[1^{k-2r}]} = {\rm char}(\bigwedge^k V) = s_{1^k}$ by basic properties of the decomposition of exterior powers under the action of the symplectic group. Since $i_{\bO}$ and $i_{\Sp}$ are algebra homomorphisms, we see that $\bH_\nu = \det(\bh_{\nu_i - i +j})$ gets sent to $\det(s_{1^{\nu_i - i + j}})$, which is the Schur function $s_{\nu^\dagger}$ by the Jacobi--Trudi formula \cite[\S I.3, eqn. (3.5)]{macdonald}.
\end{proof}

Now we focus on the orthogonal case (the symplectic case is almost identical). By \eqref{eqn:lwood-orth}, 
\[
s_{[\lambda]} = \sum_{\mu \in Q_1} (-1)^{|\mu|/2} \sum_\nu c^\lambda_{\mu, \nu} s_\nu.
\]
Since $c^\lambda_{\mu, \nu} = c^{\lambda^\dagger}_{\mu^\dagger, \nu^\dagger}$ (use that $s_\mu s_\nu = \sum_\lambda c^\lambda_{\mu,\nu} s_\lambda$ \cite[\S I.9]{macdonald} and the involution $\omega$ defined above), and $Q_1^\dagger = Q_{-1}$ \eqref{eqn:koszul-transpose}, we can rewrite this as
\[
s_{[\lambda^\dagger]} = \sum_{\mu \in Q_{-1}} (-1)^{|\mu|/2} \sum_\nu c^\lambda_{\mu, \nu} s_{\nu^\dagger}.
\]
In particular, the coefficient of $s_{\nu^\dagger}$ is $\sum_{\mu \in Q_{-1}} (-1)^{(|\lambda| - |\nu|)/2} c^\lambda_{\mu, \nu}$. By Lemma~\ref{lem:calc-C}, we get 
\begin{align*} \label{eqn:newconj}
s_{[\lambda^\dagger]} = \sum_{(\nu, s) \in \Gamma(\lambda, \Psi_\lambda)} (-1)^s C^\lambda_{\nu, s} s_{\nu^\dagger}.
\end{align*}
Finally, apply the involution $i_\bO$ to this equation and use Lemma~\ref{lem:schurH} to get \eqref{eqn:conj}. The last part of the theorem follows directly from Lemma~\ref{lem:schurH} and \eqref{eqn:duality}.

\end{document}